\documentclass[english,10pt]{amsart}
\usepackage[english]{babel}

\usepackage[matrix,arrow]{xy}
\xyoption{all}
\usepackage{amscd,amssymb,amsfonts,amsmath}

\usepackage{graphics}
\usepackage{epsfig}
\usepackage{mathrsfs}

\usepackage{array}

\newcommand\mypagesizel{
\textwidth= 6.5in
\textheight=9in
\voffset-.55in
\hoffset -0.75in
\marginparwidth=56pt
}

\newcommand{\p}[0]{{\mathbb P}}

\renewcommand{\phi}{\varphi}

\newcommand{\sF}{\mathscr{F}}

\newcommand{\sO}{\mathscr{O}}

\mypagesizel

\newtheorem{thm}{Theorem}[section]

\newtheorem{lemma}[thm]{Lemma}

\newtheorem{prop}[thm]{Proposition}

\newtheorem*{thm*}{Theorem}

\theoremstyle{definition}

\newtheorem{const}[thm]{Construction}

\newtheorem{defn-thm}[thm]{Definition-Theorem} 
\newtheorem{defn-lemma}[thm]{Definition-Lemma}

\theoremstyle{remark}
\newtheorem{rem}[thm]{Remark}

\newtheorem*{not-and-def}{Notation and definitions}
\newtheorem{exmp}[thm]{Example}

\numberwithin{equation}{section}

\begin{document}

\title[]{SINGULAR RATIONALLY CONNECTED SURFACES WITH NON-ZERO PLURI-FORMS}

\author{Wenhao OU} 

\address{Wenhao OU: Institut Fourier, UMR 5582 du
  CNRS, Universit\'e Grenoble 1, BP 74, 38402 Saint Martin
  d'H\`eres, France} 

\email{}

\subjclass[2010]{}

\begin{abstract}
Let $X$ be a projective rationally connected surface with canonical singularities carrying a non-zero reflexive pluri-forms, \textit{i.e.} the reflexive hull of $(\Omega_X^1)^{\otimes m}$ has a non-zero global section for some positive integer $m$. We show that any such surface X can be obtained from a rational ruled surface by a very explicit sequence of blow-ups and blow-downs. Moreover, we interpret the existence of non-zero pluri-forms in terms of semistable reduction.
\end{abstract}

\maketitle

\tableofcontents

\section{Introduction}
\label{Introduction}

Recall that a projective variety $X$ is said to be rationally connected if for any two general points in $X$, there
exists a rational curve passing through them, see \cite[Def. 3.2 and Prop. 3.6]{Kol96}. It is known that for a smooth
projective rationally connected variety $X$, $H^0(X,$  $(\Omega_X^1)^{\otimes m})=\{0\}$ for $m>0$, see
\cite[Cor. IV.3.8]{Kol96}. In \cite[Thm. 5.1]{GKKP11}, it is shown that if a pair $(X,D)$ is klt and $X$ is rationally
connected, then $H^0(X,\Omega_X^{[m]})=\{0\}$ for $m>0$, where $\Omega_X^{[m]}$ is the reflexive hull of $\Omega_X^m$.
By \cite[Thm. 3.3]{GKP12}, if $X$ is factorial, rationally connected and with canonical
singularities, then $H^0(X,(\Omega_X^1)^{[\otimes m]}) = \{0\}$ for $m>0$, where $(\Omega_X^1)^{[\otimes m]}$ is the reflexive hull of
$(\Omega_X^1)^{\otimes m}$. However, this is not true without the assumption of being factorial, see \cite[Example 3.7]{GKP12}. In this paper, our aim is to classify rationally connected surfaces with
canonical singularities which have non-zero reflexive pluri-forms. We will give two method to construct such surfaces (see Construction \ref{main-const} and Construction \ref{const-quo-thm}) and we will also prove that every such surface can be constructed by both of these methods (see Theorem \ref{constr-thm} and Theorem \ref{quo-thm}). This gives an affirmative answer to  \cite[Remark and Question 3.8]{GKP12}

The following example is the one given in \cite[Example 3.7]{GKP12}.

\begin{exmp} 
\label{first-exa}
Let $\pi ' : X' \to \mathbb{P}^1$ be any smooth ruled surface. Choose four distinct points $q_1$, $q_2$, $q_3$, $q_4$ in $\mathbb{P}^1$. For each point $q_i$, perform the following sequence of birational
transformations of the ruled surface:

\begin{enumerate}
\item[(i)] Blow up a point $x_i$ in the fibre over $q_i$. Then we get two $(-1)$-curves which meet transversely at
$x_i'$.
\item[(ii)] Blow up the point $x_i'$. Over $q_i$, we get two disjoint $(-2)$-curves and one $(-1)$-curve. The
$(-1)$-curve appears in the fibre with multiplicity two.
\item[(iii)] Blow down the two $(-2)$-curves. We get two singular points on the fibre, each of them is of type $A_1$.
\end{enumerate}
In the end, we get a rationally connected surface $\pi : X\to \mathbb{P}^1$ with canonical singularities such that
$H^0(X,(\Omega_X^1)^{[\otimes 2]}) \neq \{0\}$.
\end{exmp}

We will prove that every projective rationally connected surface $X$ with canonical singularities and having
non-zero pluri-forms can be constructed by a similar method (see Construction \ref{main-const} below) from a smooth ruled surface over $\mathbb{P}^1$. 

\begin{const}
\label{main-const}
Take a smooth ruled surface $ X_0 \overset{\pi_0}{\longrightarrow} \mathbb{P}^1$ and choose distinct points $q_1,..., q_r $ in
$\mathbb{P}^1$ with $r \geqslant 4$. We perform a sequence of birational transformations as follows.
\begin{enumerate}
\item[(i)] For each $q_i$, perform the same sequence of birational transformations as in Example \ref{first-exa}. We
get a fibre surface $\pi_1 : X_1 \to \mathbb{P}^1$. The non-reduced fibres of $\pi_1$ are $\pi_1^*q_1,...,\pi_1^*q_r$.
\item[iii)] Perform finitely many times this birational transformation: blow up a smooth point on a non-reduced
fibre and then blow down the strict transform of the initial fibre. We obtain another fibre surface $p : X_f \to \mathbb{P}^1$ (see Lemma \ref{contract-1}).
\item[(iii)] Starting from $X_f$, perform a sequence of blow-ups of smooth points, we get a surface $X_{a}$.
\item[(iv)] Blow down some chains of exceptional $(-2)$-curves for $X_a \to X_{f}$ (this is always possible, see \S \ref{Proof of Theorem 1.3}), we obtain a rational surface $X$.
\end{enumerate}
\end{const}

\centerline{
\xymatrix{
X_a \ar[r] \ar[rd]_{\mathrm{blow-ups}} & X \ar[d]^f \\
 & X_f \ar@{-->}[r] \ar[rd]_p & X_1 \ar[r] \ar[d]^{\pi_1} & X_0 \ar[ld]^{\pi_0}\\
 & & \p^1
}
}

\begin{thm}
\label{constr-thm}
The surface obtained by Construction \ref{main-const} is a rationally connected surface which carries non-zero pluri-forms. Conversely, if $X$ is a projective rationally connected surface with canonical singularities such that $H^0(X, (\Omega^1_X)^{[\otimes m]})
\neq \{0\}$ for some $m>0$, then $X$ can be constructed by the method described in Construction \ref{main-const}.
\end{thm}

Note that if $X$ is a rational surface obtained by Construction \ref{main-const}, then there is a fibration $\pi:X \to \mathbb{P}^1$ induced by $\pi_0$. This fibration has multiple fibres  over the points $q_1,..., q_r$ that we have chosen at the beginning of the construction. In fact, these multiple fibres are exactly the source of non-zero forms on $X$ by the theorem below.

\begin{thm}
\label{iso-pluri-forms}
Let $X$ be a projective rationally connected surface with canonical singularities and having non-zero  pluri-forms. If $X_f$ is the result of a MMP, then $X_f$ is a Mori fibre space over $\mathbb{P}^1$. Let $p:X_f \to \mathbb{P}^1$ be the fibration. If $r$ is the number of points over which $p$ has non-reduced fibres, then we have $r \geqslant 4$
and 
\begin{center}
$H^0(X,(\Omega_{X}^1)^{[\otimes m]}) \cong H^0(X_f,(\Omega_{X_f}^1)^{[\otimes m]}) \cong H^0(\mathbb{P}^1,
\mathscr{O}_{\mathbb{P}^1}(-2m+[\dfrac{m}{2}]r))$
\end{center} 
for $m>0$. In particular, for fixed $m$, the number of $m$-pluri-forms depends only on the number of multiple fibres.
\end{thm}

We note that both in Theorem \ref{iso-pluri-forms} and in Construction \ref{main-const}, we meet a
surface named $X_f$. We will see later (in the proof of Theorem \ref{constr-thm} in \S \ref{Proof of Theorem 1.3}) that, by choosing a good MMP, these two surfaces are identical. The points $q_1,...,q_r$ are exactly the points over which $p:X_f\to \p^1$ has multiple fibres. By the semistable reduction, we can find a Galois cover $\gamma: E\to \p^1$ with $E$ such that $Z \to E$ has only reduced fibres, where $Z$ is the normalisation of $X_f\times_{\p^1} E$. Let $Y$ be the normalisation of $X\times_{\p^1} E$. The following theorem shows that we can always choose a finite Galois cover $\gamma$ which has degree $4$ and the pluri-forms on $X$ are exactly the $G$-invariant pluri-forms on $Y$, where $G$ is the Galois group of $\gamma$.

\begin{thm}
\label{quo-thm}
Let $X$ be a projective rationally connected surface with canonical singularities and having non-zero pluri-forms. Let $\pi$ be the composition of $X\to X_f \to \p^1$. Then there is a commutative diagram 

\centerline{
\xymatrix{
Y \ar[rr]^{\Gamma}_{4:1 \ \mathrm{cover}}\ar[d]_{\pi'} & &  X\ar[d]^{\pi} \\
E \ar[rr]^{\gamma}_{4:1  \ \mathrm{cover}} & & \mathbb{P}^1
}
}
\noindent such that  $E$ is a smooth curve of positive genus and $Y$ is a projective surface with canonical singularities. Both $\gamma$ and $\Gamma$ are Galois covers with Galois group $G:=\mathbb{Z}/2\mathbb{Z} \times \mathbb{Z}/2\mathbb{Z}$ and $\Gamma$ is \'etale in codimension $1$. Moreover, for all $m\geqslant 0$, we have $H^0(X,(\Omega_{X}^1)^{[\otimes m]}) \cong H^0(Y,(\Omega_{Y}^1)^{[\otimes m]})^G \cong H^0(E,(\Omega_{E}^1)^{\otimes m})^G.$
\end{thm}

Note that $Y$ is not rationally connected since $E$ is not rationally connected. This theorem shows that every projective rationally connected surface with canonical singularities which has non-zero pluri-forms can be constructed by the following method.

\begin{const}
\label{const-quo-thm}
Let $Y$ be a projective surface with canonical singularities and let $G$ be a finite subgroup of Aut$(Y)$ whose action is \'etale in codimension 1. Assume that there is a $G$-invariant fibration $\pi'$ from $Y$ to a smooth curve $E$ of positive genus such that $E/G = \mathbb{P}^1$ and that general fibres of $\pi'$ are smooth rational curves. Let $X=Y/G$. Then $X$ is rationally connected (see \cite[Thm. 1.1]{GHS03}) and $H^0(X, (\Omega_X^1)^{[\otimes m]}) \neq \{0\}$ for some $m>0$. 
\end{const}

\section{Notation and outline of paper}

Throughout this paper, we will work over $\mathbb{C}$, the field of complex numbers. Unless otherwise specified, every
variety is an integral $\mathbb{C}$-scheme of finite type. A curve is a variety of dimension 1 and a surface is a
variety of dimension 2. For a variety $X$, we denote the sheaf of K\"{a}hler differentials by $\Omega_{X}^1$. Denote
$\bigwedge^p\Omega_{X}^1$ by $\Omega_{X}^p$ for $p\in \mathbb{N}$.

For a coherent sheaf $\mathscr{F}$ on a variety $X$, we denote by $\mathscr{F}^{**}$ the reflexive hull of
$\mathscr{F}$. There is an important property for reflexive sheaves.

\begin{lemma}[{\cite[Prop. 1.6]{Har80}}]
\label{refle-cod-2}
Let $\mathscr{F}$ be a coherent sheaf on a normal variety $V$. Then $\mathscr{F}$ is reflexive
if and only if $\mathscr{F}$ is torsion-free and for each open $U \subseteq X$ and each closed subset $Y\subseteq U$ of
codimension at least 2, $\mathscr{F}(U) \cong j_*\mathscr{F}(U\setminus Y)$, where $j:U\setminus Y \to U$ is the
inclusion map.
\end{lemma}

Let $V$ be a normal variety and let $V_{0}$ be its smooth locus. We denote a canonical divisor by $K_V$. Moreover, let
$\Omega_{V}^{[p]}$ (resp. $(\Omega_{V}^1)^{[\otimes p]}$) be the reflexive hull of $\Omega_{V}^p$ (resp.
$(\Omega_{V}^1)^{\otimes p}$). By Lemma \ref{refle-cod-2}, it is the push-forward of the locally free sheaf
$\Omega_{V_{0}}^p$ (resp. $(\Omega_{V_{0}}^1)^{\otimes p}$) to $V$ since $V$ is smooth in codimension 1.

Let $S$ be a normal surface. A smooth rational curve $C$ in $S$ is a $(-k)$-curve if $S$ is smooth along $C$ and the intersection number $C\cdot C =-k$. A projective birational morphism $r : \widetilde{S} \to S$ is called the minimal resolution of singularities (or minimal resolution for short) if $\widetilde{S}$ is smooth and $K_{\widetilde{S}}$ is $r$-nef. There is a unique minimal resolution of singularities for a normal surface and any resolution of singularities factors through the minimal resolution.

Let $S$ be a normal surface and let $r : \widetilde{S} \to S$ be the minimal resolution of singularities of $S$. We say that $S$ has \textit{canonical singularities} if the intersection number $K_{\widetilde{S}} \cdot C$ is zero for every $r$-exceptional curve $C$. Canonical surface singularities are also called Du Val singularities. We know all of these singularities, they are $A_i$, $D_j$, $E_k$ where $i \geqslant 1$, $j \geqslant 3$ and $k=6,7,8$. For more details on Du Val singularities, see \cite[\S  4.1]{KM98}.

Let $p: V \to B$ be a fibration from a normal variety to a smooth curve. If the non-reduced fibres of $p$ are
$p^*z_1$,..., $p^*z_r$, then  the \textit{ramification divisor}  of $p$ is the divisor defined by $$R=p^*z_1+ \cdots +p^*z_r - \mathrm{Supp}\, (p^*z_1+ \cdots +p^*z_r).$$

Let $X$ be a projective rationally connected surface with canonical singularities which carries non-zero pluri-forms. Then we can run a minimal model program for $X$ (for more details on MMP, see \cite[\S 1.4 and \S 3.7]{KM98}). We obtain a sequence of divisorial contractions  $$X=X_0 \to X_1 \to \cdots \to X_n.$$ Since $K_X$ is not pseudo-effective, neither is $K_{X_n}$. Thus $X_n$ is a Mori fibre space. we have a Mori fibration $p : X_n \to B$. Therefore we have two possibilities: either $\mathrm{dim}\, B=0$ or $\mathrm{dim}\,B=1$. If $\mathrm{dim}\,B=0$, then $X_n$ is a Fano variety with Picard number $1$. Here, a Fano variety is a normal projective variety whose anti-canonical divisor is an ample $\mathbb{Q}$-Cartier divisor. In \S 2, we will prove that $X$ can not have any non-zero pluri-form in this case. Hence we only need to deal with the case that $\mathrm{dim}\,B=1$. In \S 3, we will study some properties for Mori fibre surfaces over a curve. In the last three sections, we will prove Theorem \ref{iso-pluri-forms}, \ref{constr-thm} and \ref{quo-thm} in this order.

\section{Vanishing theorem for Fano varieties with Picard number 1}
\label{Vanishing theorem for Fano varieties with Picard number 1}

The aim of this section is to prove the following theorem.

\begin{thm}
\label{fib-over-curve}
Let $V$ be a $\mathbb{Q}$-factorial klt Fano variety  with Picard number $1$. Then $H^0(V,(\Omega_{V}^1)^{[\otimes m]})=\{0\}$
for any $m>0$.
\end{thm}

Before proving the theorem, we recall the notion of slopes. Let $V$ be a normal projective
$\mathbb{Q}$-factorial variety of dimension $d$. Let $A$ be an ample divisor in $V$. Then for a coherent sheaf  
$\mathscr{F}$, we can define $\mu_A(\mathscr{F})$ the slope of $\mathscr{F}$ with respect to $A$ by
\begin{displaymath}
\mu_A(\mathscr{F}) := \frac{\mathrm{det}\, \mathscr{F} \cdot A^{d-1}}{\mathrm{rank}\, \mathscr{F}},
\end{displaymath}  
where $\mathrm{det}\, \mathscr{F}$ is the reflexive hull of $\bigwedge^{\mathrm{rank}\, \mathscr{F}}\mathscr{F}$. Moreover, let 
\begin{center}
$\mu^{max}_A(\mathscr{F}) = \mathrm{sup}\{\mu_A(\mathscr{G}) \ |\ \mathscr{G} \subseteq \mathscr{F} \ \mathrm{a \ 
coherent \ subsheaf}\}$.
\end{center}

For any coherent sheaf $\mathscr{F}$, there is a saturated coherent subsheaf $\mathscr{G} \subseteq \mathscr{F}$ such
that  $\mu^{max}_A(\mathscr{F}) = \mu_A(\mathscr{G})$, see \cite[Prop. A.2]{GKP12}. 

\begin{prop}
\label{decomp-slope}
Let $V$ be a  projective normal variety which is $\mathbb{Q}$-factorial, let $H$ be an ample divisor in $V$. Then
for any  two coherent sheaves $\mathscr{F}$ and $\mathscr{G}$ on $V$,
\begin{center}
$\mu^{max}_H((\mathscr{F}\otimes \mathscr{G})^{**}) = \mu^{max}_H(\mathscr{F}) + \mu^{max}_H(\mathscr{G})$.
\end{center}
\end{prop}

For a proof of this proposition, see \cite[Prop. A.14]{GKP12}. Now we are ready to prove Theorem
\ref{fib-over-curve}.

\begin{proof}[Proof of Theorem \ref{fib-over-curve}]

We may assume that $\mathrm{dim}\,V>1$. We will argue by contradiction. Assume that there is a positive integer $m$ such that $H^0(V,(\Omega_{V}^1)^{[\otimes m]}) \neq \{0\}$.
Let $H$ be an ample divisor on $V$.

Since $H^0(V,(\Omega_{V}^{1})^{[\otimes m]}) \neq \{0\}$ for some $m>0$, we have an injective morphism of sheaves 
$$\mathscr{O}_{V}\hookrightarrow (\Omega_{V}^1)^{[\otimes m]}.$$ This shows that $$\mu^{max}_H((\Omega_{V}^1)^{[\otimes m]}) \geqslant
\mu_H(\mathscr{O}_{V})=0.$$ By Proposition \ref{decomp-slope}, we have $\mu^{max}_H(\Omega_{V}^{[1]}) = m^{-1}
\mu^{max}_H((\Omega_{V}^1)^{[\otimes m]}) \geqslant 0$.

Therefore, there is a non-zero saturated coherent sheaf $\mathscr{F} \subseteq \Omega_{V}^{[1]}$ such that
$\mu_H(\mathscr{F}) \geqslant 0$.  Observe that $\mathrm{rank}\, \mathscr{F}< \mathrm{dim}\,V$, otherwise $\mathscr{F}=\Omega_{V}^{[1]}$ and det$\, \mathscr{F} \cong K_{V}$. Thus $\mu_H(\mathscr{F}) < 0$, a contradiction. 

We have two possibilities, either $\mu_H(\mathscr{F})> 0$ or $\mu_H(\mathscr{F}) = 0$.

\textit{Case 1.} Assume that $\mu_H(\mathscr{F})> 0$. Since $V$ has Picard number $1$, det$(\mathscr{F})$ is
ample and its Kodaira-Iitaka dimension is $\mathrm{dim}\,V$. However this
contradicts Bogomolov-Sommese vanishing theorem (see \cite[Cor. 1.3]{Gra12}). 

\textit{Case 2.} Assume that $\mu_H(\mathscr{F})= 0$. Let $\mathscr{G}=$det$\, \mathscr{F}$. Then $\mathscr{G} \cdot H^{(\mathrm{dim}\,V-1)}=0$. Since $V$ is $\mathbb{Q}$-factorial and klt, by \cite[Lem. 2.6]{AD12}, there exists an integer $l$ such that $(\mathscr{G}^{\otimes l})^{**}$ is isomorphic to $\mathscr{O}_V$. Let $m$ be the smallest positive integer such that $(\mathscr{G}^{\otimes m})^{**} \cong \mathscr{O}_V$. We can construct the cyclic cover $q:Z \to V$ of $V$ corresponding to $\mathscr{G}$, see \cite[Def. 2.52]{KM98}. Then $(q^*\mathscr{G})^{**} \cong \mathscr{O}_Z$. Since $q$ is \'etale in codimension 1, $Z$ is also klt by \cite[Prop. 3.16]{Kol97} and $-K_Z=q^*(-K_{V})$ is ample. Thus $Z$ is rationally connected by \cite[Cor. 1.3 and 1.5]{HM07}. And there are natural injective morphisms $$ (q^*\mathscr{G})^{**} \hookrightarrow (q^* \Omega_{V}^{[\mathrm{rank\, \mathscr{F}}]})^{**} \hookrightarrow \Omega_{Z}^{[\mathrm{rank\, \mathscr{F}}]}.$$ Hence we have an injection $\mathscr{O}_Z \hookrightarrow \Omega_{Z}^{[\mathrm{rank\, \mathscr{F}}]}$, but this contradicts \cite[Thm. 5.1]{GKKP11}.
\end{proof}

\section{Mori fibre surfaces over a curve}
\label{Mori fibre surfaces over a curve}

Recall that a Mori fibration $V\to W$ is a projective fibration such that $-K_V$ is relatively ample and the relative Picard number is $1$. A Mori fibre space $V$ is just a variety endowed with a Mori fibration $V\to W$. In this section, we study Mori fibration from a quasi-projective surface with canonical singularities to a smooth curve. In the first subsection, we will give some properties of the fibres. In the second subsection, we will classify the singularities on a non-reduced fibre.

We would like to introduce some notation for this section first. Let $p: S\to B $ be a Mori fibration, where $B$ is a
smooth curve and $S$ is a normal surface with canonical singularities. Let $r: \widetilde{S} \to S$ be the
minimal resolution and $\tilde{p}=p \circ r : \widetilde{S} \to B$. 

Since $S$ is singular at only finitely many points, $p$ is smooth over general points of $B$ and general fibres are all isomorphic to $\mathbb{P}^1$. Note that
a point in a smooth curve can also be regarded as a Cartier divisor and since any two fibres of $p$ are numerically equivalent, we have $K_{S} \cdot p^*z=-2$ and $p^*z \cdot p^*z=0$ for any $z\in B$ by the adjunction formula.

We recall the definition of dual graph. Let $E= \bigcup E_i$ be a collection of proper curves on a normal surface $V$ such that $V$ is smooth along $E$. The
\textit{dual graph} $\Gamma$ of $E$ is defined as follow:

(1) The vertices of $\Gamma$ are the curves $E_i$.

(2) Two vertices $E_i \neq E_j$ are connected with $(E_i \cdot E_j)$ edges.

\subsection{Some properties of fibres}\label{Some properieties of fibres}

Running a $\tilde{p}$-MMP for $\widetilde{S}$, we obtain a sequence of divisorial contraction

\centerline{
\xymatrix{
\widetilde S\ar[r]\ar[ddrrr]^{\tilde p} & Y_1\ar[r] & \cdots\ar[r] & Y_{n'}\ar[dd]\\
& & & \\
& & & B\\
}
}

\begin{lemma}
\label{Hirzebruch surface}
With the notation in the diagram above, the surface $Z=Y_{n'}$ is a ruled surface over $B$. Moreover, the support of $\tilde{p}^*b$ is a snc tree, \textit{i.e.} it is a snc divisor and its dual graph is a tree, for all point $b\in B$.
\end{lemma}

\begin{proof}
Since $p_{Z}: Z\to B$ is the result of a  $\tilde{p}$-relative MMP,  $Z$ is a smooth surface. Note that $K_{Z}$ has negative intersection number with general fibres of $p_{Z}$. Hence the next extremal contraction in the MMP is a contraction of fibre type.  This contraction gives $Z$ the ruled surface structure over $B$.
 
Note that $\widetilde{S}$ can be obtained by a sequence of blow-ups from $Z$. Thus the dual graph of the support of any fibre of  $\tilde{p}$ is a snc tree.
\end{proof}

We collect some properties for the fibre of $p:S\to B$.

\begin{prop}
\label{prop-fibre}
Let $z$ be a point in $B$. Then 
\begin{enumerate}
\item[(1)] the support $C$ of $p^*z$ is an irreducible Weil divisor for every $z\in B$;
\item[(2)] the coefficient of $C$ in $p^*z$ is at most equal to 2;
\item[(3)] $S$ is smooth along the support of $p^*z$ if and only if $p^*z$ is reduced;
\item[(4)] There exist at most 2 singular points of $S$ on $C$.
\end{enumerate}
\end{prop}

\begin{proof}
(1) Assume the opposite and let $D$, $D'$ be two distinct components in $p^*z$ which meet. Then $D \cdot D'>0$ and $D \cdot D <0$ since $p^*z \cdot D =0$. However, there is a positive number $\lambda$ such that $\lambda D$ and $D'$ are numerically equivalent by the definition of Mori fibration. Hence $D \cdot D>0$. This is a contradiction.  

(2) Let  $\alpha \in \mathbb{N}$ be the coefficient of $C$ in $p^*z$. Then $$-2=K_S \cdot p^*z=\alpha K_S \cdot C.$$ However, since $K_S$ is a Cartier divisor, $K_S \cdot C \in \mathbb{Z}$. Thus $-2 \in \alpha \mathbb{Z}$ which means $\alpha \leqslant 2$. 

(3) Note that $S$ is Cohen-Macaulay since it is a normal surface. Therefore the Cartier divisor $p^*z$ is also Cohen-Macaulay. Hence it is generically reduced if and only if it is a reduced subscheme. Moreover, since $B$ is a smooth curve, the morphism $p$ is a flat morphism. 

First we assume that $p^*z$ is reduced. Then the arithmetic genus of $p^*z$ is $0$ since $p$ is flat and general fibres of $p$ are smooth rational curves. Hence $p^*z$ is isomorphic to $\p^1$ (cf. \cite[Ex. IV.1.8(b)]{Har77}). Hence $p$ is smooth over $z$ since it is flat. Thus $S$ is smooth along $p^*z$.

Conversely, assume that $S$ is smooth along $p^*z$. Then by adjunction formula, we have $$2h^1(C,\sO_C)-2=(K_S+C)\cdot C=K_S \cdot C<0.$$ Therefore, $K_S \cdot C=-2$ which is equal to $K_S\cdot p^*z$. Hence $p^*z$ is  reduced. 

(4) Assume that $S$ is not smooth along $C$. Then $p^*z=2C$ by (2) and (3). Let  $\widetilde{C}$ be the strict transform of $C$ in $\widetilde{S}$ and let $E=\tilde{p}^*z-2\widetilde{C}$. Since   $E$ is $r$-exceptional, we have $$K_{\widetilde{S}} \cdot E= r^*K_S \cdot E=0.$$  Thus $$K_{\widetilde{S}} \cdot \widetilde{C} = 2^{-1}(K_{\widetilde{S}} \cdot \tilde{p}^*z)=-1.$$ By the adjunction formula $\widetilde{C}^2=-1$ ($\widetilde{C}$ is smooth by Lemma \ref{Hirzebruch surface}). Then $$-1=\widetilde{C}^2 = 2^{-1} \widetilde{C} \cdot (\tilde{p}^*z - E)=-2^{-1}\widetilde{C} \cdot E.$$ We obtain $ \widetilde{C}\cdot E=2$. This implies that  $\widetilde{C}$ and $E$ meet at most at two points. Hence $S$ has at most two singular point on $C$.
\end{proof}

\subsection{Singularities on non-reduced fibres}
\label{Singularities on non-reduced fibres}
The aim of this subsection is to  give a complete list of possible multiple fibres of $p:S \to B$. The subject was studied in \cite[\S 11.5]{KM99}, but we will give some elementary proofs of the results here. In the remaining of this section, we will assume that $p$ has non-reduced fibre over and only over $0 \in B$. By Proposition \ref{prop-fibre}, this implies that $p^*0=2C$, where $C$ is the support of $p^*0$. We denote the strict transform of $C$ in $\widetilde{S}$ by $\widetilde{C}$. We will prove the following theorem.

\begin{thm}
\label{table-fibre}
Let $p:S \to B$ be  a Mori fibration such that $S$ is a quasi-projective surface with canonical singularities and  $B$ is a smooth curve. Assume that $p^*0$ is a multiple fibre, where $0\in B$. Let $r: \widetilde{S} \to S$ be the minimal resolution of singularities along the fibre $p^*0$ and let $\tilde{p}=p\circ r$. We have the following table of possible singular fibres for $p^*0$. In the table, the dual graph is the one of the support of $\tilde{p}^* 0 \subseteq \widetilde{S}$, the point with label $s$ corresponds to $\widetilde{C}$ and the other points correspond to the $r$-exceptional divisors.

\begin{center}
\renewcommand{\arraystretch}{0}
    \begin{tabular}{ |>{\centering\arraybackslash}m{3cm}|>{\centering\arraybackslash}m{7cm}|} 
    
    \hline
         
  \vspace{0.1cm} \centering  $\mathrm{Type\ of\ fibre}$ \vspace{0.1cm} & $\mathrm{Dual\ graph}$ \\ \hline
   \centering $(A_1+A_1)$ & \vspace{0.1cm} \centerline{
\xymatrix{
\overset{1}{\bullet} \ar@{-}[r] &\overset{s}{\circ} \ar@{-}[r] &\overset{2}{\bullet}
}} \\ \hline
   \centering $(D_3)$  & \vspace{0.1cm} \centerline{
\xymatrix{
\overset{1}{\bullet} \ar@{-}[r] &\overset{3}{\bullet} \ar@{-}[r] \ar@{-}[d] &\overset{2}{\bullet}\\
 & \underset{s}{\circ} 
}}  
\\ \hline
  \centering  $(D_i) \ \mathrm{with} \ i>3 $ & \vspace{0.1cm}  \centerline{
\xymatrix{
\overset{1}{\bullet} \ar@{-}[r] &\overset{3}{\bullet} \ar@{-}[r] \ar@{-}[d] &\overset{4}{\bullet} \ \ \cdots \ \ \overset{i}{\bullet} \ar@{-}[r] &\overset{s}{\circ}\\
 & \underset{2}{\bullet} 
}}      \\
    \hline
    \end{tabular}
\end{center}
\end{thm}

In the table of the theorem above, we see that a multiple fibre of type $(A_1+A_1)$ is a multiple fibre which contains exactly two singular points and both of them are of type $A_1$. A multiple fibre of type $(D_i)$ with $i\geqslant 3$ is a multiple fibre which contains exactly one singular point which is of type $D_i$.

We will prove the theorem by proving several lemmas (Lemma \ref{2-sing-fib}, \ref{1-sing-fib-1} and \ref{constr-fib}). Note that by Proposition \ref{prop-fibre}, there exist one or two singular points of $S$ on $C$. We will first treat the case of two singular points.

\begin{lemma}
\label{2-sing-fib}
Assume that there are two singular points on $C$, then each of them is of type $A_1$.
\end{lemma}

\begin{proof}
Let $E=\tilde{p}^*0-2\widetilde{C}$. As in the proof of Proposition \ref{prop-fibre}.4, we have $\widetilde{C} \cdot E=2$. Since there are two singular points on $C$, $E$ has exactly two connected components. Thus we can decompose $E$ into $D+D'+R$, where $D$ and $D'$ are the two components in $E$ which meet $\widetilde{C}$. Then we have $$\widetilde{C} \cdot D = \widetilde{C} \cdot D'=1,\ D \cdot D' =0 \ \mathrm{and} \ \widetilde{C} \cdot R =0.$$ Note that both $D$ and $D'$ are $(-2)$-curves, hence $$0=\tilde{p}^*z \cdot D = 2\widetilde{C} \cdot D + D^2 +D' \cdot D + R \cdot D = R\cdot D.$$ This implies that $R$ and $D$ do not meet since both $R$ and $D$ are effective. By symmetry, $R$ and $D'$ do not meet neither. However, since $E$ has exactly two connected components, we obtain that $R=0$. Hence both of the singular points on $C$ are of type $A_1$.
\end{proof}

This type of fibre is the type $(A_1+A_1)$. Note that this type of fibre does exist by Example \ref{first-exa}. Next we will study the case of one singular point. We will prove that this isolated singularity is of type $D_i$ ($i \geqslant 3$ and the type $D_3$ is just $A_3$).

\begin{lemma}
\label{1-sing-fib-1}
The isolated singularity on the fibre over $0 \in B$ can only be of type $D_i$ $(i \geqslant 3)$.
\end{lemma}

\begin{proof}

Let $C_0=\widetilde{C}$ and let $E_0=\tilde{p}^*0 - 2C_0$. As in the proof of Proposition \ref{prop-fibre}.4, we have $$C_0^2=-1 \ \mathrm{and} \  E_0 \cdot C_0 = 2.$$ Since $E_0 \cdot p^*0=0$, we obtain $E_0^2=-4$.  

Since there is only one singular point on $C$ and the support of $\tilde{p}^*0$ is a snc tree (see Lemma \ref{Hirzebruch surface}), we can decompose $E_0$ into $2C_1+E_1$, where $C_1$ is the unique component in $E_0$ which meets $C_0$. Then $C_1$ is a $(-2)$-curve. Since $2C_1 \cdot \tilde{p}^*0 = 0$ and $E_0^2=-4$, we have $$ E_1^2=-4 \ \mathrm{and} \ C_1 \cdot E_1=2.$$ Thus the support of $E_1$ intersects $C_1$ at one or two points. If they intersect at two points, then as in Lemma \ref{2-sing-fib}, $E_1=D+D'$ where $D$, $D'$ are smooth rational curves, and we have $$D\cdot D'=0, \  D\cdot C_1=1, \ D'\cdot C_1=1.$$ If $E_1$ and $C_1$ intersect at one points, then we can decompose $E_1$ into $2C_2+E_2$ where $C_2$ is the unique component of $E_1$ which meets $C_1$. As above, we have $$C_2^2=-2,\ E_2^2=-4 \ \mathrm{and} \ E_2 \cdot C_2 = 2.$$ We are in the same situation as before. Hence by induction, we can decompose $E_0$ into $2(D_1+\cdots +D_i)+D+D'$ where $D$, $D'$ and all the $D_j$'s are $(-2)$-curves. Furthermore, we have $$D \cdot D'=0, \ D_i \cdot D=1,\ D_i \cdot D'=1,\ D_j \cdot D_{j+1}=1$$ for $1 \leqslant j \leqslant i-1$, and $D_j \cdot D_k=0$ if $k-j>1$. This shows that the singular point is of type $D_{i+2}$.
\end{proof}

These types of fibres are the type $(D_i)$ $(i \geqslant 3)$. Now we will prove that these kinds of fibres exist. We will need the following lemma.

\begin{lemma}
\label{contract-1}
Let $x \in S$ be a smooth point over $0 \in B$ and let $W$ be the blow-up of $S$ at $x$ with exceptional divisor $E
\subseteq W$. Let $D$ be the strict transform of $C$ in $W$. Then we can blow down $D$ and obtain another Mori fibre surface  $q:T \to B$.
\end{lemma}

\begin{proof}
Let $W \hookrightarrow W_1$ be a projective compactification of $W'$ such that $W_1$ has canonical singularities. If we can blow down $D$ in $W_1$, then we can also blow down $D$ in $W$. Hence we may assume that $W$ is projective.

We have $C \cdot C=0$, $K_{S} \cdot C=-1$, $K_{W} \cdot E=-1$, $E \cdot E=-1$ and $D \cdot E=1$. Thus $$K_{W} \cdot D=0\ \mathrm{and}\  D \cdot D=-1.$$ 

Let $H$ be an ample divisor on $W$. Then there is a positive integer $k$ such that $(H+kD)\cdot D=0$. Let $A=H+kD$. Note that $A$ is nef and big and $D$ is the only curve which has intersection number $0$ with $A$. Since $K_W\cdot D=0$, for large enough positive integer $a$, the divisor $aA-K_X$ is nef and big. Hence by the basepoint-free theorem (see \cite[Thm. 3.3]{KM98}), there is a positive integer $b$ such that the linear system $|bA|$ is basepoint-free. Let $c:W\to T$ be the fibration induced by the linear system $|bA|$. Then $c$ contracts exactly $D$. Since $D$ is contracted by $W\to B$,  the fibration $W\to B$ induces a fibration  $q:T \to B$ which is also a Mori fibration.
\end{proof}

We can use the elementary transformation in Lemma \ref{contract-1} to construct every type of multiple fibres mentioned above.

\begin{lemma}
\label{constr-fib}
If $S$ is of type $(A_1+A_1)$ over $0 \in B$ then $T$ is of type $(D_3)$ over $0 \in B$.
If $S$ is of type $(D_i)$ over $0 \in B$ then $T$ is of type $(D_{i+1})$ over $0 \in B$ for $i \geqslant 3$.
\end{lemma}

\begin{proof}
We will compute the dual graph of the support of the fibre $\tilde{q}^*0$, where $\widetilde{T}\to T$ is the minimal resolution and $\tilde{q} $ is the composition of $\widetilde{T} \to T\to B$. Let $W$ be the same as in Lemma \ref{contract-1}. From the construction of $T$, we know that $\widetilde{T} \to T$ factors through $\widetilde{T} \to W$ and the last morphism is also the minimal resolution of $W$. Since $W\to S$ is a blow-up of a smooth point of $S$, the surface $\widetilde{T} $ can be obtained by blowing up the same point in $\widetilde{S}$.

If the fibre $p^*0$ is of type $(A_1+A_1)$, the dual graph of the support of $\tilde{p}^*0$ in $\widetilde{S}$ is

\centerline{
\xymatrix{
\overset{1}{\bullet} \ar@{-}[r] &\overset{s}{\circ} \ar@{-}[r] &\overset{2}{\bullet}
}}
\noindent where $s$ represents $\widetilde{C}$. Blow up the point we mentioned above, the new graph is 

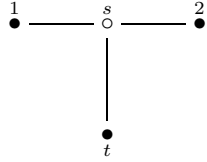
\centerline{
\xymatrix{
\overset{1}{\bullet} \ar@{-}[r] &\overset{s}{\circ}\ar@{-}[d] \ar@{-}[r] &\overset{2}{\bullet}\\
& \underset{t}{\bullet} 
}}

This graph is the dual graph of the support of $\tilde{q} ^*0$ and the point with label $t$ corresponds to the strict transform of the support of $q^*0$ in $\widetilde{T}$.  The graph shows that there is only one singular point of $T$ on $q^*0$ which is of type $D_3$. Hence the fibre $q^*0$ is of type $(D_3)$.

If $p^*0$ is of type $D_i$, then from the proof of Lemma \ref{1-sing-fib-1}, we know that the dual graph of the support of $\tilde{p}^*0$ is

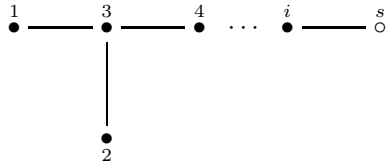
\centerline{
\xymatrix{
\overset{1}{\bullet} \ar@{-}[r] &\overset{3}{\bullet}\ar@{-}[d] \ar@{-}[r] &\overset{4}{\bullet} \ \  \cdots \ \   \overset{i}{\bullet} \ar@{-}[r]& \overset{s}{\circ}  \\
& \underset{2}{\bullet} 
}}

\noindent where the point with label $s$ correspond to $\widetilde{C}$ (If $i=3$, then $s$ is just connected to the point with label $3$). By blowing up the point, we obtain the dual graph of the support of $\tilde{q}^*0$, which is 

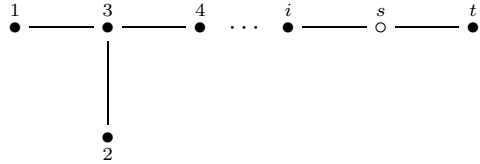
\centerline{
\xymatrix{
\overset{1}{\bullet} \ar@{-}[r] &\overset{3}{\bullet}\ar@{-}[d] \ar@{-}[r] &\overset{4}{\bullet} \ \  \cdots \ \   \overset{i}{\bullet} \ar@{-}[r]& \overset{s}{\circ} \ar@{-}[r] &\overset{t}{\bullet} \\
& \underset{2}{\bullet}
}}

The point with label $t$ corresponds to the strict transform of the support of $q^*0$ in $\widetilde{T}$. This implies that $q^*0$ is of type $(D_{i+1})$.
\end{proof}

\begin{proof}[{Proof of Theorem \ref{table-fibre}}]
We can deduce the theorem from Lemma \ref{2-sing-fib}, \ref{1-sing-fib-1} and \ref{constr-fib}.
\end{proof}

Now we will show that every singular fibre can be obtained from a smooth fibre by the methods we mentioned above.

\begin{lemma}
\label{all-constr}
The singular fibre $p^*0$ of $p:S\to B$ can be obtained from a smooth ruled surface  $S_1\to B$ by the method of Example \ref{first-exa} followed by a finite sequence of elementary birational transformations described in Lemma \ref{contract-1}.
\end{lemma}

\begin{proof}
Let $\widetilde{S}\to Z$ be the result of a $\tilde{p}$-MMP. Then $Z\to B$ is a ruled surface by Lemma \ref{Hirzebruch surface}. Moreover, $\widetilde{S}$ can be obtained from $Z$ by a sequence of blow-ups.

If $p^*0$ is of type $(A_1+A_1)$, then $\widetilde{S}$ can be obtained from $Z$ by two blow-ups as in the first two step of Example \ref{first-exa}. Blow-down the two $(-2)$-curves in  $\widetilde{S}$\, we obtain $S$. In this case, we  take $S_1=Z$.

If $p^*0$ is of type $(D_i)$, then the dual graph of $\tilde{p}^*0$ is 

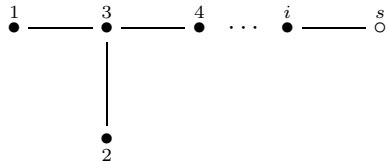
\centerline{
\xymatrix{
\overset{1}{\bullet} \ar@{-}[r] &\overset{3}{\bullet}\ar@{-}[d] \ar@{-}[r] &\overset{4}{\bullet} \ \  \cdots \ \   \overset{i}{\bullet} \ar@{-}[r]& \overset{s}{\circ}  \\
& \underset{2}{\bullet} 
}}

Note that the curve corresponding to the point $s$ is a $(-1)$-curve. Hence we may blow down this curve and the curves which correspond to the points in the graph above which do not meet the point  $s$ (This is always possible by the lemma below). Then we will obtain another Mori fibre surface $p_U:U \to B$. The fibre $p_U^*0$ is of type $(D_{i-1})$ if $i>3$ and of type $(A_1+A_1)$ if $i=3$. Moreover, $U$ is smooth around the image of the curve corresponding to $s$. If we perform the birational transformation in Lemma \ref{contract-1} for $U$, then we will obtain $S$. 

By induction, we can conclude the lemma.
\end{proof}

The following lemma shows that we can contract some connected collection of $(-2)$-curves in a  surface.

\begin{lemma}
\label{contra-chain}
Let $E = \bigcup_{1\leqslant k \leqslant i} E_k$ be a connected collection of $(-2)$-curves in a smooth surface $V$ whose dual graph is the same  as the one of the support of the exceptional exceptional set of a minimal resolution for a canonical surface singularity. Then there exists a morphism $c: V \to W$ such that $W$ has canonical singularities and $c$ contracts exactly $E$.
\end{lemma}

\begin{proof}
We have $K_V \cdot E_k=0$ for every $k$.  The intersection matrix $\{E_k \cdot E_j\}$ is negative definite by \cite[Lem. 3.40]{KM98}. Thus there is a  contraction $c: V \to W$ contracting exactly $E$ by \cite[Prop 4.10]{KM98}. Note that $c$ is also the minimal resolution of $W$ and $K_V=c^*K_{W}$. Hence $W$ has canonical singularities.
\end{proof}

\section{Proof of Theorem 1.4}
\label{Proof of Theorem 1.4}

We will first prove Theorem \ref{iso-pluri-forms}. Let $X$ be a projective rationally connected   surface with canonical singularities which has non-zero pluri-forms. Run a MMP for $X$. We will get a sequence of divisorial contractions

\begin{center}
$X=X_0 \to X_1 \to \cdots \to X_n=X_f$.
\end{center}
The rational surface $X_f$ is a Mori fibre surface over $\p^1$ by Theorem \ref{fib-over-curve}. Let $p:X_f \to \mathbb{P}^1$ be the Mori fibration. Let $f:X\to X_f$ be the composition of the sequence of the birational morphisms above and let $\pi = p \circ f : X \to \mathbb{P}^1 $. Then for any $m \in \mathbb{N}$, there is an injection $H^0(X,(\Omega_{X}^1)^{[\otimes m]}) \hookrightarrow H^0(X_f,(\Omega_{X_f}^1)^{[\otimes m]})$.

\subsection{Source of non-zero reflexive pluri-forms}
\label{Source of nonzero reflexive pluri-forms}

In this subsection, we will find out the source of non-zero pluri-forms on $X_f$.  Let $U$ be the smooth locus of $X_f$. Then the morphism of locally free sheaves on $U$ $$p^*\Omega_{\p^1}^1\to \Omega_{U}^1$$ factors through $$\phi: p^*\Omega_{\p^1}^1 \otimes \sO_U(R) \to \Omega_U^1$$ where $R$ is the ramification divisor of $p$. Let $V$ be the largest subset of $U$ such that for any point $x\in V$, the valuation of $\phi$ at $x$ is injective. Then $\mathrm{codim}\, X_f\backslash V \geqslant 2$.  By Lemma \ref{refle-cod-2}, this implies that $H^0(X_f,(\Omega_{X_f}^1)^{[\otimes m]}) \cong H^0(V,(\Omega_{V} ^1)^{\otimes m})$ for any $m \in \mathbb{N}$.

Consider the exact sequence of sheaves on $V$

\begin{center}
$0 \to p^* \Omega_{\mathbb{P}^1}^1 \otimes \mathscr{O}_V(R)  \to \Omega_{V}^1
\to\mathscr{G} \to  0$
\end{center} 
where $\mathscr{G}$ is isomorphic to $\Omega_{V/\mathbb{P}^1}^1/\mathrm{torsion}$. It is an invertible sheaf on $V$ since $\mathscr{G} \otimes k_x$ is of rank $1$ at every point $x$ of $V$, where $k_x$ is the residue field of $x$ (see \cite[Ex. II.5.8]{Har77}). Then there is a filtration over $V$ $$(\Omega_{V}^1 )^{\otimes m} = \mathscr{F}_0 \supseteq \mathscr{F}_{1} \supseteq \cdots \supseteq \mathscr{F}_m \supseteq \mathscr{F}_{m+1}= 0$$ such that $\mathscr{F}_{i} / \mathscr{F}_{i+1}$ is the direct sum of copies of $\mathscr{G}^{\otimes (m-i)} \otimes (p^* \Omega_{\mathbb{P}^1}^1 \otimes \mathscr{O}_V(R))^{\otimes i}$ for every $i \in \{0,...,m\}$ and $\sF_m \cong (p^* \Omega_{\mathbb{P}^1}^1 \otimes \mathscr{O}_V(R) )^{\otimes m} \cong (p^* \Omega_{\mathbb{P}^1}^1)^{\otimes m} \otimes \mathscr{O}_V(mR)$.

\begin{lemma}
\label{source}
With the notation above, there is a natural isomorphism 
\begin{center}
$H^0(X_f,(\Omega_{X_f}^1)^{[\otimes m]}) \cong H^0(\mathbb{P}^1, \mathscr{O}_{\mathbb{P}^1}(-2m) \otimes p_{*} \mathscr{O}_{X }(mR))$
\end{center}
for all $m\geqslant 0$.
\end{lemma}

\begin{proof} 
Fix some $m\geqslant 0$. For a general point $z \in \mathbb{P}^1$, the support $C$ of the fibre $p^*z$ is isomorphic to $\mathbb{P}^1$ and is contained in $V$. Since $p$ is smooth along $C$, we have $$\mathscr{G}|_{C} \cong \mathscr{O}_{C}(-2) \ \mathrm{and} \ (p^*\Omega_{\mathbb{P}^1}^1 \otimes \mathscr{O}_V(R))|_{C} \cong \mathscr{O}_{C}.$$  Thus $(\mathscr{F}_{i} / \mathscr{F}_{i+1})|_{C}$ is the direct sum of copies of $\mathscr{O}_{C}(2(i-m))$ for $i \leqslant m$. Hence $H^0(V, \mathscr{F}_{i} /
\mathscr{F}_{i+1})=0$ and $H^0(V, \mathscr{F}_{i}) \cong H^0(V, \mathscr{F}_{i+1})$ for $i < m$. This implies that $$H^0(V, (\Omega_V^1)^{\otimes m}) \cong H^0(V, (p^* \Omega_{\mathbb{P}^1}^1)^{\otimes m} \otimes \mathscr{O}_V(mR) ).$$

By Lemma \ref{refle-cod-2}, the isomorphism above induces an isomorphism $$H^0(X, (\Omega_X^1)^{[\otimes m]}) \cong H^0(X, (p^* \Omega_{\mathbb{P}^1}^1)^{\otimes m} \otimes \mathscr{O}_X(mR)).$$

Note that the right hand side above is isomorphic to $H^0(\mathbb{P}^1, p_*((p^* \Omega_{\mathbb{P}^1}^1)^{\otimes m} \otimes \mathscr{O}_X(mR)))$. By the projection formula, it is isomorphic to $H^0(\mathbb{P}^1, ( \Omega_{\mathbb{P}^1}^1)^{\otimes m} \otimes p_*\mathscr{O}_X(mR))$. Hence $$H^0(X_f,(\Omega_{X_f}^1)^{[\otimes m]}) \cong H^0(\mathbb{P}^1, \mathscr{O}_{\mathbb{P}^1}(-2m) \otimes p_{*} \mathscr{O}_{X }(mR)).$$
\end{proof}

Note that $p_{*} \mathscr{O}_{X }(mR)$ is a torsion-free sheaf of rank 1 on $\mathbb{P}^1$. Thus it is an invertible sheaf and there is a $k \in \mathbb{Z}$ such that $\mathscr{O}_{\mathbb{P}^1}(k)$ is isomorphic to $p_{*} \mathscr{O}_{X }(mR)$. In the following lemma we will compute the integer $k$.

\begin{lemma}
\label{pluri-form-Xf}
Assume that the non-reduced fibres of $p:X_f \to \mathbb{P}^1$ are over $z_1,...,z_r$. Then for $m\in \mathbb{N}$, we have $p_* \mathscr{O}_{X}(mR) \cong \mathscr{O}_{\mathbb{P}^1}([\dfrac{m}{2}](z_1+...+z_r)) \cong
\mathscr{O}_{\mathbb{P}^1}([\dfrac{m}{2}]r)$, where $[\quad]$ is the integer part. In particular, $H^0(X_f,(\Omega_{X_f}^1)^{[\otimes m]}) \cong H^0(\mathbb{P}^1, \mathscr{O}_{\mathbb{P}^1}(-2m+[\dfrac{m}{2}]r))$.
\end{lemma}

\begin{proof}
Since the problem is local around every point $z_i$, we may assume that $r=1$ for simplicity. From Proposition \ref{prop-fibre}.1 and Proposition \ref{prop-fibre}.2, we know that $R$ is irreducible and $p^*z_1 = 2R$. We may assume that $p_* \mathscr{O}_{X}(mR) \cong \mathscr{O}_{\mathbb{P}^1}(k\cdot z_1)$ and we have to prove that $k=[\dfrac{m}{2}]$.

Note that $\gamma \in H^0(\mathbb{P}^1, \mathscr{O}_{\mathbb{P}^1}(k\cdot z_1))$ is just a rational function on $\mathbb{P}^1$ which can only have pole at $z_1$ with multiplicity at most $k$. Its pull-back to $X$ is a rational function which can only have pole along $R$ with multiplicity at most $2k$. Thus $k$ is the largest integer
such that $2k \leqslant m$, \textit{i.e.} $k=[\dfrac{m}{2}]$.
\end{proof}

\subsection{Back to the initial variety}\label{Back to the initial variety}
We have studied $X_f$ and now we have to reverse the MMP and pull back  pluri-forms to the initial variety $X$.   Our aim is to prove that $$H^0(X,(\Omega_{X}^1)^{[\otimes m]})
\cong H^0(X_f,(\Omega_{X_f}^1)^{[\otimes m]}).$$

We will need the following proposition. 

\begin{prop}
\label{iso-sing}
Let $S$ be a projective surface which has at most canonical singularities. Let $c: S \to T$ be a divisorial contraction in a MMP. Let $E$ be the exceptional divisor and let $x$ be the image of $E$. Then $T$ is smooth at $x$.
\end{prop}

\begin{proof}
We suppose the opposite. Let $r_S: \widetilde{S} \to S$ and $r_{T}: \widetilde{T} \to T$ be the minimal resolutions. Let $\widetilde{E}$ be the strict transform of $E$ in $\widetilde{S}$. We have a commutative diagram

\centerline{
\xymatrix{
\widetilde{S} \ar[r]^{\tilde{c}} \ar[d]_{r_S}  & \widetilde{T} \ar[d]^{r_{T}}\\
S \ar[r]^c  & {T}\\
}
}
 Then $K_{\widetilde{S}} \cdot \widetilde{E} = r^*K_S \cdot \widetilde{E} = K_S \cdot {r_S}_*\widetilde{E} = K_S \cdot E < 0$ by the definition of MMP. Since $K_{{\widetilde{T}}}$ is $r_T$-nef, $\widetilde{E}$ must be contracted by $\tilde{c}$. Since $E$ is over $x$, $\tilde{c}(\widetilde{E})$ is contained in an exceptional divisor $D$ of $r_{T}$. Let $\widetilde{D}$ be the strict transform of $D$ in $\widetilde{S}$. Then $\widetilde{D}$ is contracted by $r_S$ for $\widetilde{D} \neq \widetilde{E}$ and the image of $\widetilde{D}$ in $T$ is a point. Thus $\widetilde{D}$ is a $(-2)$-curve in $\widetilde{S}$ since $S$ has canonical singularities. 
 
Since $T$ has canonical singularities, $D$ is also a $(-2)$-curve. Note that $\tilde{c}$ is the composition of a sequence of blow-ups of smooth points (see \cite[Cor. V.5.4]{Har77}). Moreover, for $\tilde{c}$, we have to blow up the point $x$ which is contained in $D$. Hence the self-intersection number of $\widetilde{D}$ is less than $(-2)$. We obtain a contradiction.
\end{proof}

By Proposition \ref{iso-sing}, every exceptional divisor of $f:X \to X_f$ is over a smooth point of $X_f$. Now we can prove the isomorphism we mentioned at the beginning of this subsection. 

\begin{lemma}
\label{entire-ramifi}
The natural injection $H^0(X,(\Omega_{X}^1)^{[\otimes m]})
\to H^0(X_f,(\Omega_{X_f}^1)^{[\otimes m]})$ is an isomorphism. 
\end{lemma}

\begin{proof}
Let $X_a\to X$ be a projective birational morphism which is the minimal resolution for the singular points of $X$ lying over smooth points of $X_f$. Then there is a natural injection $$H^0(X_a,(\Omega_{X_a}^1)^{[\otimes m]})  \to  H^0(X,(\Omega_{X}^1)^{[\otimes m]}).$$ By Proposition \ref{iso-sing}, $f^{-1}$ is an isomorphism around the singular points of $X_f$. Hence all exceptional divisors of $X_a \to X_f$ are over smooth points of $X_f$. This implies that $X_a$ can be obtain from $X_f$ by a sequence of blow-ups of smooth points (see \cite[Cor. V.5.4]{Har77}). 

\centerline{
\xymatrix{
 & X_a \ar[ld]_{\mathrm{resolution}} \ar[rd]^{\mathrm{blow-up}} &\\
X \ar[rr] & & X_f
}
}

Then we have a natural isomorphism $ H^0(X_a,(\Omega_{X_a}^1)^{[\otimes m]}) 
\cong  H^0(X_f,(\Omega_{X_f}^1)^{[\otimes m]}),$ which implies that $H^0(X,(\Omega_{X}^1)^{[\otimes m]})
\cong H^0(X_f,(\Omega_{X_f}^1)^{[\otimes m]})$.
\end{proof}

We can conclude Theorem \ref{iso-pluri-forms}.

\begin{proof}[{Proof of Theorem \ref{iso-pluri-forms}}]
By Theorem \ref{fib-over-curve}, we have a Mori fibration $p:X_f\to \p^1$. Lemma \ref{pluri-form-Xf} and \ref{entire-ramifi} shows that $H^0(X,(\Omega_{X}^1)^{[\otimes m]}) \cong H^0(X_f,(\Omega_{X_f}^1)^{[\otimes m]}) \cong H^0(\mathbb{P}^1,
\mathscr{O}_{\mathbb{P}^1}(-2m+[\dfrac{m}{2}]r)).$
\end{proof}

\section{Proof of Theorem 1.3}\label{Proof of Theorem 1.3}
We will  prove Theorem \ref{constr-thm} in this section. If $X$ is a projective rationally connected surface with canonical singularities such that $H^0(X, (\Omega_X)^{[\otimes m]}) \neq \{0\}$ for some $m>0$ and   $X_f$ is the result of a MMP, then $X$ and $X_f$ are isomorphic around the singular locus of $X_f$ by Proposition \ref{iso-sing}. The proof of Lemma \ref{entire-ramifi} gives us an idea of how to reconstruct $X$ from $X_f$. First we construct the surface $X_{a}$ (the surface defined in the proof of Lemma \ref{entire-ramifi}) which can be obtained from $X_f$ by a sequence of blow-ups of smooth points. Then we blow down some exceptional $(-2)$-curves for $X_{a} \to X_f$ and we obtain $X$. Note that these are just the birational transformations mentioned in step (iii) and (iv) of Construction \ref{main-const}.

In order to contract the $(-2)$-curves in the transformation above, we want to use Lemma \ref{contra-chain}. Thus, we have to study the structure of the exceptional set of $X_{a} \to X_f$.

\begin{lemma}
\label{fib-sturc-b-u}
Denote a germ of smooth surface by $(0\in S)$. Let $h:S'\to S$ be the composition of a sequence of blow-ups of smooth
points over $0 \in S$. Let $D$ be the support of $h^*0$. Then any $(-2)$-curve in $D$ meets at most 2 other
$(-2)$-curves. In another word, the dual graph of $D$ cannot contain a subgraph as below such that each vertex of the
subgraph corresponds a $(-2)$-curve. 

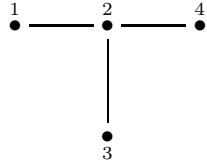
\centerline{
\xymatrix{
\overset{1}{\bullet} \ar@{-}[r] & \overset{2}\bullet \ar@{-}[r] \ar@{-}[d] & \overset{4}\bullet\\
& \underset{3}\bullet
}
}
\end{lemma}

\begin{proof}
Assume the opposite. We know that we can reverse the process of blow-ups of smooth points by running a MMP relatively to $S$. Thus these four curves will be successively contracted during the MMP. The first one contracted cannot be the curve corresponding to the point with label $2$, since after the contraction, the dual graph of the remaining curves is a tree by an analogue result of Lemma \ref{Hirzebruch surface}. Without lose of generality, we may assume that the curve corresponding to the point with label $1$ is the first one  contracted.

If the curve corresponding to the point $3$ (or $4$) is contracted secondly, then the self-intersection number of the curve corresponding to the point $2$ becomes at least $0$. If the curve corresponding to the point $2$ is contracted secondly, a further contraction will also produce a curve with self-intersection number at least $0$. 

However, this curve of self intersection at least $0$ is over $0\in S$, it must have negative self-intersection number by the negativity theorem (see \cite[Lem. 3.40]{KM98}). This leads to a contradiction.
\end{proof}

In particular, by the lemma above, every connected collection of $(-2)$-curves in $D$ has a dual graph as below

\centerline{
\xymatrix{
{\bullet} \ar@{-}[r] & \bullet    \ \ \cdots \ \ \bullet 
}
}

This is the dual graph of the exceptional set of the minimal resolution for the singularity of type $A_i$. By Lemma \ref{contra-chain}, it is possible to contract such a chain of $(-2)$-curves.

Now we can prove Theorem \ref{constr-thm}.

\begin{proof}[{Proof of Theorem \ref{constr-thm}}]
First  let $X$ be a projective rationally connected surface with canonical singularities which carries non-zero pluri-forms. We will prove that $X$ can be constructed  by the method of Construction \ref{main-const}. Let $f:X\to X_f$ be the result of a MMP and let $X_a$ be the surface defined in the proof of Lemma \ref{entire-ramifi}. The surface $X$ can be obtained from $X_a$ by a contraction of chains of $(-2)$-curves by Lemma \ref{entire-ramifi}, \ref{fib-sturc-b-u} and \ref{contra-chain}. By the proof of Lemma \ref{entire-ramifi}, $X_a$ can be obtained from $X_f$ by a sequence of  blow-ups of smooth points. Since $X_f\to \p^1$ is a Mori fibration and $X_f$ has canonical singularities, $X_f$ can be obtain from a smooth ruled surface $X_0\to \p^1$ by the method of step (i) and (ii) of Constriction \ref{main-const} (see Lemma \ref{all-constr}). Thus $X$ can be constructed  by the method of Construction \ref{main-const}.

Now, let $X$ be a surface constructed by the method of Construction \ref{main-const}. We will prove that $X$ carries non-zero pluri-forms. If $X_f$ is the surface described in Construction \ref{main-const}, then there is a natural morphism $f:X \to X_f$ and both $X$ and $X_f$ have canonical singularities (see Lemma \ref{all-constr} and \ref{contra-chain}). Moreover, by construction, all of the $f$-exceptional curves are over smooth points of $X_f$. Let $Y\overset{g}{\longrightarrow} X_f$ be the result of a $f$-MMP for $X$. 

\centerline{
\xymatrix{
X \ar[r]  \ar[rd]_f & Y \ar[d]^g  \\
 & X_f
}
}

Then $K_Y$ is $g$-nef since $g$ is birational. Moreover, all of the $g$-exceptional curves are over smooth points of $X_f$. Let $r_Y:\widetilde{Y} \to Y$ be the minimal resolution. Then $K_{\widetilde{Y}}=r_Y^*K_Y$ since $Y$ has canonical singularities. If $h$ is the composition of $\widetilde{Y} \to Y \to X_f$, then $K_{\widetilde{Y}}$ is $h$-nef. Thus $h$ is the minimal resolution of $X_f$ and $h^{-1}$ is an isomorphism on the smooth locus of $X_f$. Hence $g^{-1}$ is also an isomorphism on the smooth locus of $X_f$. This implies that $g:Y\to X_f$ is an isomorphism. 

Since $X_f\to \p^1$ is a Mori fibration, $X_f$ is isomorphic to the result of a MMP for $X$ (This is why we use the same notation $X_f$). After Lemma \ref{pluri-form-Xf}, we know that $X_f$ carries non-zero pluri-forms. By Lemma \ref{entire-ramifi}, this shows that $X$ carries non-zero pluri-forms.
\end{proof}

\section{Proof of Theorem 1.5}\label{Proof of Theorem 1.5}

We would like to prove Theorem \ref{quo-thm} in this section. In \cite[Remark and Question 3.8]{GKP12}, for $X$ in
Example \ref{first-exa}, we can find a smooth elliptic curve $E$, a smooth ruled surface $Y$ (which is $\widetilde{X}$
in \cite{GKP12}) such that $\mathbb{P}^1$ is the quotient of $E$ by $\mathbb{Z}/2\mathbb{Z}$ and $X$ is the quotient of
$Y$ by the same group. In this section, we would like to construct such a surface $Y$ for any rationally connected
surface $X$  with canonical singularities and having non-zero pluri-forms.

We will first construct the curve $E$.

\begin{prop}
\label{constr-curve}
Let $q_1,...,q_r$ be $r$ different points on $\mathbb{P}^1$ with $r \geqslant 4$, then there exist a smooth curve $E$, a $4:1$ Galois cover $\gamma: E \to \mathbb{P}^1$ with Galois group $G=\mathbb{Z}/2\mathbb{Z} \times \mathbb{Z}/2\mathbb{Z}$ such that $\gamma$ is exactly ramified  over the $q_i$'s and the degrees of ramification are all equal to $2$.
\end{prop}

\begin{proof}
Since $r \geqslant 4$, we can find an elliptic curve $D$ and a $2:1$ cover $\alpha : D \to \mathbb{P}^1$ such that
$\alpha$ is ramified exactly over $q_1,q_2,q_3,q_4$.  Let $\alpha^{-1}(\{q_i\})=\{s_{i},t_{i}\}$ for $i > 4$ and let $\alpha^{-1}(\{q_i\})=\{s_i\}$ for $i=1,2$. 

If $r>4$, then $\mathscr{O}_{D}(2(r-4)s_1)$ is isomorphic to  $\mathscr{O}_{D}(\sum_{i>4}s_{i}+\sum_{i>4}t_{i})$. Thus we can construct a ramified $2:1$ cyclic cover of $E$, with respect to the line bundle $\mathscr{O}_{D}((r-4)s_1)$, $$\beta : E\to
D$$ such that $E$ is smooth and $\beta$ ramified exactly over $\{s_{i},t_{i}\ | \ i > 4 \}$  (see
\cite[Def. 2.50]{KM98}). 

If $r=4$, then $\mathscr{O}_{D}(2s_1-2s_2) \cong \mathscr{O}_{D}$ and
we can construct a $2:1$ cyclic cover of $E$, with respect to the non-trivial invertible sheaf $\mathscr{O}_{D}(s_1-s_2)$, $$\beta : E\to D$$ such that $E$ is a smooth elliptic curve and $\beta$ is \'etale.

Finally, in both of the case above, the composition $$\gamma = \alpha \circ \beta: E \to \mathbb{P}^1$$ is a $4:1$ cover which is exactly ramified  over the $q_i$'s and the degrees of ramification are all equal to $2$. Moreover $\mathbb{P}^1$ is the quotient $E/G$ where $G=\mathbb{Z}/2\mathbb{Z} \times \mathbb{Z}/2\mathbb{Z}$.
\end{proof}

\begin{rem}
What we want in the lemma above is to construct a finite morphism $\gamma:E\to \p^1$ which is exactly ramified over the $q_i$'s and all of the ramified degrees are equal to $2$. Note that the finite cover $\gamma$ we constructed above is of degree four and the one in \cite[Remark and Question 3.8]{GKP12} is of degree two. However, if $r$ is odd, then the Hurwitz's theorem (see \cite[Cor. IV.2.4]{Har77}) shows that it is not possible to have a $2:1$ cover which satisfies the condition.  
\end{rem}

Now we will prove Theorem \ref{quo-thm}.

\begin{proof}[{Proof of Theorem \ref{quo-thm}}]
Let $q_1,...,q_r$ be all of the points in $\p^1$ over which $p:X_f\to \p^1$ has multiple fibres. Let $\gamma:E\to \p^1$ be $4:1$ cover constructed in Lemma \ref{constr-curve}. Let $Z$ be the normalisation of the fibre product $X_f\times_{\p^1} E$. Let $q:Z\to E$ and $\Gamma_f:Z\to X_f$ be the natural projections. Then $\Gamma_f$ is \'etale over the smooth locus of $X_f$ and $q$ has only reduced fibres.

We know that we can reconstruct $X$ from $X_f$ (see \S \ref{Proof of Theorem 1.3}). Since $\Gamma_f: Z \to X_f$ is \'etale over the smooth locus of $X_f$, every
operation we do with $X_f$ can be done in the analogue way with $Z$. After the operations, the surface $Y\overset{g}{\longrightarrow} Z$ we obtained is just the normalisation of  $X\times_{\p^1} E$. We have a commutative diagram as below

\centerline{
\xymatrix{
Y \ar@/_1pc/[dd]_{\pi'} \ar[r]^{\Gamma} \ar[d]^g & X \ar@/^1pc/[dd]^{\pi} \ar[d]_f  \\
Z \ar[r]^{\Gamma_f} \ar[d]^q & X_f \ar[d]_p   \\
E \ar[r]^{\gamma} & \p^1
}
}

Then $\Gamma$ is \'etale over the smooth locus of $X$ and $X=Y/G$ where $G$ is the Galois group of $\gamma$. The sheaf $\Gamma_*(\Omega_Y^{[\otimes m]})$ is a $G$-sheaf on $X$ (the action of $G$ on $X$ is trivial) which is reflexive (see \cite[Prop. 1.7]{Har80}). Then $(\Gamma_*(\Omega_Y^{[\otimes m]}))^G$ is also reflexive (see \cite[Lem. B.4]{GKKP11}) and is isomorphic to $\Omega_X^{[\otimes m]}$ since $\Gamma$ is \'etale  over the smooth locus of $X$. Thus we have $$H^0(Y,(\Omega_{Y}^1)^{[\otimes m]})^G \cong H^0(X,(\Omega_{X}^1)^{[\otimes m]}).$$

Moreover, for any $m\geqslant 0$, the natural morphism  $$H^0(Y,(\Omega_{Y}^1)^{[\otimes m]}) \to H^0(Z,(\Omega_{Z}^1)^{[\otimes m]})$$ is an isomorphism by the same argument as in the proof of Lemma \ref{entire-ramifi}. Since every fibre of $q$ is reduced and general fibre of $q$ are smooth rational curves,  by the same argument as Lemma \ref{source}, we have
$$H^0(Y,(\Omega_{Y}^1)^{[\otimes m]}) \cong H^0(E,(\Omega_{E}^1)^{\otimes m}).$$

Hence we obtain isomorphisms $$H^0(X,(\Omega_{X}^1)^{[\otimes m]}) \cong H^0(Y,(\Omega_{Y}^1)^{[\otimes m]})^G \cong H^0(E,(\Omega_{E}^1)^{\otimes m})^G.$$
\end{proof}

Now we want to compute the dimension of $H^0(X,(\Omega_{X}^1)^{[\otimes m]})$ in function of multiple fibres of $X_f\to \p^1$ with the formula above. We will first prove the following lemma.

\begin{lemma}
\label{dir-im-ram}
Let $R_{\gamma}$ be the ramification divisor of the finite morphism $\gamma : E \to \mathbb{P}^1$, then
$(\gamma_*\mathscr{O}_E(R_{\gamma}))^G \cong \mathscr{O}_{\mathbb{P}^1}$. 
\end{lemma}

\begin{proof}
We have $H^0(U,(\gamma_*\mathscr{O}_E(R_{\gamma}))^G) \cong H^0(\gamma^{-1}(U), \mathscr{O}_E(R_{\gamma}))^G$ for any open
set $U\subseteq \mathbb{P}^1$. Let $\theta$ be a rational function on $E$ such that $\theta$ represents an non-zero element in $H^0(\gamma^{-1}(U), \mathscr{O}_E(R_{\gamma}))^G$. Since $\theta$ is $G$-invariant, it can also be regarded as a rational function on $U$. Since $\theta$ can only have simple poles at the support of $R$ on $\gamma^{-1}(U)$, it cannot have any pole on $U$. Thus $(\gamma_*\mathscr{O}_E(R_{\gamma}))^G \cong \mathscr{O}_{\mathbb{P}^1}$.  
\end{proof}

With the notation in the proof of Theorem \ref{quo-thm}, we have $$(\Omega_E^1)^{\otimes m} \cong (\gamma^*\mathscr{O}_{\mathbb{P}^1}(-2m+[\dfrac{m}{2}](q_1+\cdots +q_r))) \otimes
\mathscr{O}_E((m-2[\dfrac{m}{2}])R_{\gamma}).$$ By the projection formula, we have $\gamma_*(\Omega_E^1)^{\otimes m} \cong \mathscr{O}_{\mathbb{P}^1}(-2m+[\dfrac{m}{2}]r)\otimes
\gamma_*\mathscr{O}_E((m-2[\dfrac{m}{2}])R_{\gamma})$. By taking the $G$-invariant part, we obtain $$(\gamma_*(\Omega_E^1)^{\otimes m})^G \cong \mathscr{O}_{\mathbb{P}^1}(-2m+[\dfrac{m}{2}]r)\otimes (\gamma_*\mathscr{O}_E((m-2[\dfrac{m}{2}])R_{\gamma}))^G.$$ The lemma above implies that $(\gamma_*(\Omega_E^1)^{\otimes m})^G \cong (\mathscr{O}_{\mathbb{P}^1}(-2m+[\dfrac{m}{2}]r)).$ Hence  $$H^0(X,(\Omega_{X}^1)^{[\otimes m]})  \cong H^0(E,(\Omega_E^1)^{\otimes m})^G\cong H^0(\p^1,(\gamma_*(\Omega_E^1)^{\otimes m})^G) \cong H^0(\mathbb{P}^1,\mathscr{O}_{\mathbb{P}^1}(-2m+[\dfrac{m}{2}]r)).$$ We recover the same formula  as in Theorem \ref{iso-pluri-forms}

\begin{exmp}
\label{fin-exa}
We will give some examples. Let $h(m,r)$ be the dimension of $H^0(\mathbb{P}^1,\mathscr{O}_{\mathbb{P}^1}(-2m+[\dfrac{m}{2}]r))$. This is just the number  of $m$-pluri-forms as a function of the number $r$ of multiple fibres of $X_f\to\p^1$.

If $r=4$, then $h(m,4)=1$ if $m>0$ is even and $h(m,4)=0$ if $m$ is odd.

If $r=5$, then $h(2,5)=2$, $h(3,5)=0$ and   $h(m,5)>0$ if $m \geqslant 4$.

If $r \geqslant 6$, then $h(m,r)>0$ for $m\geqslant 2$.

\end{exmp}

\bibliographystyle{amsalpha}
\bibliography{bibliography}

\end{document}